\documentclass[a4paper,11pt]{article}
\usepackage{latexsym}
\usepackage{amssymb}
\usepackage{enumitem}
\usepackage{xcolor}
\usepackage{commath}
\usepackage{array}
\usepackage{amsfonts}
\usepackage{amsmath,amsthm}
\usepackage{hyperref}
\usepackage{algorithm}
\usepackage{algorithmic}
\usepackage{framed}
\usepackage{boites}
\usepackage{graphicx}
\usepackage{cases}

\newenvironment{@abssec}[1]{%
    \if@twocolumn

      \section*{#1}%
    \else

      \vspace{.05in}\footnotesize
      \parindent .2in
 {\upshape\bfseries #1. }\ignorespaces
    \fi}

    {\if@twocolumn\else\par\vspace{.1in}\fi}

\newenvironment{keywords}{\begin{@abssec}{\keywordsname}}{\end{@abssec}}

\newenvironment{AMS}{\begin{@abssec}{\AMSname}}{\end{@abssec}}

\newcommand\keywordsname{Key words}
\newcommand\AMSname{AMS subject classifications}
\newcommand\AMname{AMS subject classification}
\newcommand\restr[2]{{
\left.\kern-\nulldelimiterspace 
#1 
\vphantom{|} 
\right|_{#2} 
}}
\newtheorem{theorem}{Theorem}[section]
\newtheorem{lemma}[theorem]{Lemma}
\newtheorem{corollary}[theorem]{Corollary}
\newtheorem{proposition}[theorem]{Proposition}
\newtheorem{remark}[theorem]{Remark}
\newtheorem{definition}[theorem]{Definition}

\newtheorem{mainthm}{Theorem}

\newtheorem{mainlem}{Lemma}

\newtheorem{thm}{Theorem}

\newtheorem{lem}[thm]{Lemma}

\newcommand{\NN}{\mathbb{N}}

\newcommand{\RR}{\mathbb{R}}

\def\XXint#1#2#3{{\setbox0=\hbox{$#1{#2#3}{\int}$}
\vcenter{\hbox{$#2#3$}}\kern-.5\wd0}}

\newcommand{\link}{\mathop{\circ\kern-.35em -}}

\newcommand{\ol}{\overline}
\newcommand{\pa}{\partial}

\newcommand{\dv}{\mathop{\mathrm{div}}}

\newcommand{\gr}{\nabla}

\newcommand{\al}{\alpha}
\newcommand{\be}{\beta}
  
\newcommand{\Ga}{\Gamma}

\newcommand{\De}{\Delta}
\newcommand{\ve}{\varepsilon}
 
\newcommand{\la}{\lambda}
\newcommand{\La}{\Lambda}

\newcommand{\te}{\theta}

\newcommand{\om}{\omega}
\newcommand{\Om}{\Omega}
\newcommand{\rn}{{\mathbb{R}}^N}
\newcommand{\sg}{\sigma}

\newcommand\setbld[2]{\left\{ #1 \; :\; #2\right\}}

\newcommand{\tin}{{\text{in }}}
\newcommand{\ton}{{\text{on }}}

\newcommand{\tfor}{{\text{for }}}
\newcommand{\tforall}{{\text{for all }}}
\newcommand{\tas}{{\text{as }}}
\newcommand{\id}{{\rm Id}}
\newcommand\pp[1]{\left( #1\right)}

\newcommand{\cdottone}{{\boldsymbol{\cdot}}}

\newcommand{\C}{\mathcal{C}}
\newcommand{\cC}{\mathcal{C}}

\newcommand{\cH}{{\mathcal H}}

\newcommand{\cJ}{{\mathcal J}}

\newcommand{\cS}{{\mathcal S}}

\newcommand{\cU}{{\mathcal U}}
\newcommand{\cV}{{\mathcal V}}
\newcommand{\cX}{\mathcal{X}}

\numberwithin{equation}{section}

\title{Why are the solutions to overdetermined problems usually ``as symmetric as possible"?
}
\author{Lorenzo Cavallina
\thanks{
This research was partially supported by JSPS KAKENHI (under grant nos. JP20K22298, JP22K13935, and JP21KK0044).
}}
\date{\emph{To my Professor, Shigeru Sakaguchi, for his 65th birthday}}

\begin{document}

\maketitle

\begin{abstract}
 In this paper, we study the symmetry properties of nondegenerate critical points of shape functionals using the implicit function theorem. We show that, if a shape functional is invariant with respect to some one-parameter group of rotations, then its nondegenerate critical points (bounded open sets with smooth enough boundary) share the same symmetries. We also consider the case where the shape functional exhibits translational invariance in addition to just rotational invariance. 
Finally, we study the applications of this result to the theory of one/two-phase overdetermined problems of Serrin-type. 
En passant, we give a simple proof of the fact that, under suitable smoothness assumptions, the ball is the only nondegenerate critical point of the Lagrangian associated to the maximization problem for the torsional rigidity under a volume constraint. We remark that the proof does not rely on either the method of moving planes or rearrangement techniques.
\end{abstract}



\pagestyle{plain}
\thispagestyle{plain}

\section{Introduction: 
``\emph{The answer is the ball}''}\label{introduction}
This was a recurring joke among the students of Professor Sakaguchi's Lab. No matter the question, the answer always seemed to be the ball. Indeed, among many, the ball is the solution to the following problems:

\begin{enumerate}[label=\emph{\alph*}), nolistsep]
    \item The ball minimizes the first Dirichlet eigenvalue of the Laplacian among open sets of a given volume
(the ``Faber-Krahn inequality", \cite{Fa, Kr}, see also \cite{Ka}
). 

\item The ball maximizes the ratio $\la_2/\la_1$ of the first two Dirichlet eigenvalues of the Laplacian
(the ``Payne-P\'olya-Weinberger conjecture", proven by \cite{ashbaugh benguria}).

\item The ball maximizes the second Neumann eigenvalue of the Laplacian among open sets of a given volume
(proven by \cite{Sz} in dimension $2$,
by \cite{Wei} in any dimension).

\item The ball maximizes the second Steklov eigenvalue of the Laplacian among open sets of a given volume
(proven by \cite{Ws} in dimension $2$, by \cite{Br} in any dimension).

\item\label{isoperimetric} The ball maximizes the volume among open sets of given surface area, or equivalently, minimizes the surface area among open sets of given volume (the ``isoperimetric inequality": although this is considered to be one of the oldest problems in shape optimization one has to wait until the second half of the last century for a rigorous proof in all dimensions, see \cite{DeG, Burago Zalgaller} for some history).

\item\label{saint venant} The ball maximizes the torsional rigidity, that is, the integral of the solution to the boundary value problem
\begin{equation}\label{torsion eq}
-\De u = 1 \quad \tin \Om,\quad u=0\quad \ton \pa\Om
\end{equation}
among open sets $\Om$ of given volume (the ``Saint-Venant inequality", conjectured by \cite{SV} and proven in full generality by \cite{Ta}).

\item\label{soap bubble} The ball is the unique bounded domain with constant mean curvature (``Aleksandrov's soap bubble theorem", due to \cite{Ale1958}). 

\item\label{serrin} The ball is the unique domain $\Om$ such that the solution of \eqref{torsion eq}
also satisfies $|\gr u|\equiv const$ on $\pa\Om$ 
(``Serrin's overdetermined problem": see \cite{Se1971} for the original proof when $\pa\Om$ is of class $C^2$ and $u\in C^2(\ol\Om)$, based on the method of moving planes and a refined version of Hopf's boundary lemma; see also \cite{Wei preceding} for an alternative proof that makes use of integral identities).

\item \label{garofalotachi} The ball is the unique bounded open set $\Om$ such that $\rn\setminus\ol\Om$ is connected and the solution of the boundary value problem
\begin{equation*}
-\De u=0\quad \tin\rn\setminus \ol\Om, \quad 0\le u<1 \quad \tin \rn\setminus\ol\Om, \quad u=1\quad \ton\pa\Om
\end{equation*}
also satisfies $|\gr u|\equiv const$ on $\pa\Om$ (see \cite{reichel, garofalo sartori} for two different proofs and generalizations of the above result).
\end{enumerate}

Notice that all problems above share a common structure. They all take into consideration the critical shapes $\Om$ (in many cases, maximizers or minimizers) of some shape functional $J=J(\Om)$. In particular, we remark that the overdetermined conditions of \ref{soap bubble} and \ref{serrin} translate to the necessary conditions for the optimizers of problems \ref{isoperimetric} and \ref{saint venant} respectively.  

Another feature shared by all of the problems above is rotational invariance. Notice that, if $\Om$ is a solution to a rotationally invariant problem, then any rotation $\widetilde\Om$ of $\Om$ is also a solution to the same problem. Furthermore, if the problem only had one solution, then one would conclude that $\Om=\widetilde\Om$ and, thus, that the solution $\Om$ is radially symmetric. Unfortunately, in most applications, one deals with problems that are invariant with respect to a larger group of transformations (for instance, including translations). In principle, such problems have infinitely many solutions (whenever they have at least one) and thus the same naive reasoning cannot be applied directly to obtain radial symmetry. 
In this paper, we aim to give some necessary conditions to get ``local uniqueness" for solutions (maybe up to some group of transformations depending on the case at hand). Now, once ``local uniqueness" is shown, one can reason along the same lines as before and conclude that $\Om=\widetilde\Om$ for all sufficiently small rotations $\widetilde\Om$ of $\Om$, thus showing that the solution $\Om$ is radially symmetric. 

\section{Main results} 
As briefly mentioned in the introduction, we will consider the critical shapes $\Om$ of some shape functional $J=J(\Om)$.
The ``right" condition that implies ``local uniqueness" is given by the so-called \emph{non-degeneracy} of the critical shape $\Om$. The precise definitions of criticality and non-degeneracy will be given in section \ref{section nondegeneracy}.  

Let us first introduce some notation. 
Let $\NN$ denote the set of positive integers. For $k\in\NN\cup\{0\}$ and $\al\in [0,1]$, let $\cC^{k,\al}$ denote the collection of bounded open sets of $\rn$ ($N\ge 2$) whose boundary is of class $C^{k,\al}$ (here $\cC^{k,0}=\cC^k$ simply stands for the collection of bounded open sets of class $C^k$). 

Let us also recall the definition of a \emph{one-parameter group of rotations}. We say that a subgroup $\{R_\vartheta\}_{\vartheta\in \RR}$ of the special orthogonal group $SO(N)$ is a one-parameter group of rotations if $\vartheta\mapsto R_\vartheta$ is a continuous mapping from $\RR$ to $SO(N)$ that satisfies
\begin{equation*}
    R_{s+t}=R_s\circ R_t \quad \tforall s,t\in\RR.
\end{equation*}
Thus, in particular, $R_0=\id$ is the identity element and $\{R_\vartheta\}_{\vartheta\in\RR}$ is an abelian group.

Moreover, let $\{R_\vartheta\}_{\vartheta\in\RR}$ be a one-parameter group of rotations and let $J: \cC^{m,\al}\to \RR$ be a shape functional. We say that $J$ is invariant with respect to $\{R_\vartheta\}_{\vartheta\in\RR}$ if $J(R_\vartheta(\om))=J(\om)$ for all $\vartheta\in\RR$ and $\om\in \cC^{m,\al}$. Analogously, we say that an open set $\om\in \cC^{m,\al}$ is invariant with respect to $\{R_\vartheta\}_{\vartheta\in\RR}$ if $R_\vartheta(\om)=\om$ for all $\vartheta\in\RR$.

In what follows, we will state the main theorem of this paper, which shows the link between non-degeneracy and symmetry. 

\begin{mainthm}\label{mainthm I}
Let $m\in\NN$ and $\al\in [0,1]$. Moreover, let $J: \cC^{m,\al}\to \RR$ be a shape functional that is invariant with respect to some one-parameter group of rotations $\{R_\vartheta\}_{\vartheta\in\RR}$.
If $\Om\in \cC^{m+2,\al}$ is a non-degenerate critical point of $J$ (see Definition \ref{def nondeg1}), then $\Om$ is also invariant with respect to $\{R_\vartheta\}_{\vartheta\in\RR}$. In particular, if $J$ is invariant with respect to the whole rotation group $SO(N)$, then $\Om$ is radially symmetric.  
\end{mainthm}

The same result holds true under some natural modifications if the shape functional $J$ is also translation-invariant
. To this end, we will consider the restriction of $J$ to the space $\cC^{m,\al}_\star$, defined as
\begin{equation*}
    \cC^{m,\al}_\star:=\setbld{\om\in\cC^{m,\al}}{{\rm Bar}(\om)=0} \quad \text{for all }m\in\NN, 
\end{equation*}
where ${\rm Bar}(\om):= \frac{\int_{\om} x\ dx}{|\om|}$ is the barycenter of $\om$.
\begin{mainthm}\label{mainthm II}
Let $m\in\NN$ and $\al\in [0,1]$. Moreover, 
let $J: \cC^{m,\al}_\star\to \RR$ be a shape functional that is invariant with respect to some one-parameter group of rotations $\{R_\vartheta\}_{\vartheta\in\RR}$. 
If $\Om\in \cC^{m+2,\al}_\star$ is a non-degenerate critical point of $J$ (see Definition \ref{def nondeg2}), then $\Om$ is also invariant with respect to $\{R_\vartheta\}_{\vartheta\in\RR}$. In particular, if $J$ is invariant with respect to the whole rotation group $SO(N)$, then $\Om$ is radially symmetric.  
\end{mainthm}

\begin{remark}
Notice that rotational symmetry in the case where $\Om$ is a nondegenerate local minimizer (or maximizer) can be shown by making use of the quantitative stability results of \cite{DL2019} under some structural assumptions on the shape functional.
On the other hand, the definition of nondegenerate critical shape used in this paper is general enough to deal with saddle shapes. In particular, this is necessary in order to give a symmetry result in Theorem \ref{serrin two phase} concerning the two-phase Serrin's overdetermined problem, for which it is known that, unlike the classical one-phase setting, saddle shape solutions exist (see \cite{cava2018, CY1}).
\end{remark}

This paper is structured as follows. In section \ref{sec 3}, we introduce the concept of shape derivative. There, we also recall a famous result by Novruzi--Pierre \cite[Lemma 3.1]{structure} that states that every small enough perturbation of a set can be rewritten in terms of a perturbation of the boundary along the normal direction. In section \ref{section nondegeneracy}, we give the definitions of nondegenerate critical shape for parametrized shape functional (that may or may not be translation-invariant). Here we introduce the key Lemmas \ref{lem I}-\ref{lem II}, which, roughly speaking, state that, in a neighborhood of a given nondegenerate critical shape, the set of critical shapes forms a smooth branch, whose elements are uniquely identified by the value of the given parameter. Lemma \ref{lem I} has been proved in \cite{cava nondegenerate}, while Lemma \ref{lem II} is a new result and will be proved in section \ref{sec 5}. Section \ref{sec 6} is devoted to the proof of Theorems \ref{mainthm I}--\ref{mainthm II}. Here, we combine the results of the previous sections to give a rigorous justification of the intuitive approach given in the introduction. 
Finally, in section \ref{sec 7}, we show how these results relate to the theory of one/two-phase overdetermined problems of Serrin-type and show symmetry results that, in some sense, bridge the gap between problems \ref{saint venant} and \ref{serrin}.

\section{The theory of shape perturbation and the structure of shape derivatives}\label{sec 3}

Let $m\in\NN$ and $\al\in [0,1]$ and set
\begin{eqnarray*}
    \Theta^{m,\al}:=\setbld{\te\in C^{m,\al}(\rn,\rn)}{\norm{\te}_{C^{m,\al}}<\infty}.
\end{eqnarray*}
Notice that the function space above is a Banach space with respect to the corresponding H\"older norm. 
We can define a perturbation of an open set $\om\in\cC^{m,\al}$ as \begin{equation}\label{interior perturbation}
(\id+\te)(\om):=\setbld{x+\te(x)}{x\in\om},\quad \tfor \te\in\Theta^{m,\al}  
\end{equation}
Notice that, if $\te$ is small enough, the set defined in \eqref{interior perturbation} also belongs to $\cC^{m,\al}$ and is homeomorphic to $\om$. 
The one defined in \eqref{interior perturbation} is not the only way of perturbing a given set. Another classical way is given by the so-called normal (or Hadamard) perturbations of the boundary.  
For fixed open set $\om\in\cC^{m+1,\al}$, let $n_\om$ denote its outward unit normal vector. Moreover, let
\begin{equation}
E:C^{m,\al}(\pa\om)\to C^{m,\al}(\rn, \rn)    
\end{equation}
be a bounded linear ``extension operator" that satisfies the following 
\begin{equation*}
    \restr{(E\xi)}{\pa\om}=\xi n_\om.
\end{equation*}
For a construction of $E$, we refer to \cite[Lemma 6.38]{GT} and the remarks in \cite[Appendix]{cava nondegenerate}.
For $\xi\in C^{m+1,\al}(\pa\om)$ small enough, notice that the following set $\om_\xi$ is a well-defined element of $\cC^{m,\al}$:
\begin{equation*}
\om_\xi:= (\id+E\xi)(\om).    
\end{equation*}

It is known that, under sufficient smallness and regularity assumptions, both kinds of perturbations are equivalent, that is, they describe the same sets. This assertion is rigorously justified by the following lemma. 

\begin{lem}[Reparametrization lemma, {\cite[Lemma 3.1]{structure}}]\label{reparametrization}
Let $\om\in \cC^{m+1,\al}$ and set 
\begin{equation*}
C^{m,\al}(\pa\om,\pa\om):=\setbld{G\in C^{m,\al}(\pa\om,\rn)}{G(\pa\om)\subset \pa\om}.
\end{equation*}
Then, there exist an open neighborhood $\cU$ of $0\in\Theta^{m+1,\al}$ and a unique pair of $C^1$ functions 
\begin{equation*}
    \Psi:\cU\to C^{m,\al}(\pa\om), \quad G:\cU\to C^{m,\al}(\pa\om,\pa\om)
\end{equation*}
such that for all $\te\in\cU$,
\begin{equation*}
    (\id+\te)\circ G(\te)=\id+\Psi(\te)\ n_\om\quad \ton \pa\om.
\end{equation*}
Moreover, the Fr\'echet derivative of $\Psi$ at $\te=0$ is given by
\begin{equation*}
\Psi'(0)[\te]=\restr{\te}{\pa\om}\cdot n_\om \quad\text{for any }\te\in\Theta^{m+1,\al}.
\end{equation*}
\end{lem}
\begin{remark}[On the regularity]
Lemma \ref{reparametrization} was originally stated only for $\al=0$. The proof relied on the use of some auxiliary function $\zeta_0\in C^{m+1}(\rn,\RR)$ with support in an open set $\om_0\supset\pa\om$ that satisfies
\begin{equation*}\label{properties of zeta}
\pa\om=\setbld{x\in\om_0}{\zeta_0(x)=0}, \quad \gr\zeta_0\ne 0 \quad \ton \pa\om.
\end{equation*}
The case $\al\in(0,1]$ follows by a slight modification. Indeed, we can construct a function $\zeta_0$ of class $C^{m+1,\al}$ with the same properties by
taking $\zeta_0$ to be the oriented (or signed) distance function to $\pa\om$ (see \cite[Remark 3.3]{structure} and \cite[Theorem 8.2, (iii)]{SG}). 
The rest of the proof then follows verbatim from that of \cite[Lemma 3.1]{structure}.
\end{remark}

\begin{proposition}\label{G is bij}
For $\te\in\cU$ small enough, the map $G(\te): \pa\om\to\pa\om$ is a bijection.
\end{proposition}
\begin{proof}
We will first show that there exists a natural number $n\in \NN$ such that if $\norm{\theta}<1/n$ then the map $G(\theta): \pa\om\to\pa\om$ is surjective. Assume, for sake of contradiction, that for all $n\in\NN$ there exists an element $\theta_n\in\Theta$ with $\norm{\theta}<1/n$ but such that $G(\theta_n)$ is not surjective. In particular, there exists a sequence of points $x_n\in\pa\om$ such that 
\begin{equation}\label{no surj}
x_n\notin G(\theta_n)(\pa\om).    
\end{equation}
Since $\pa\om$ is compact, we can extract a subsequence of $x_n$ converging to some point $x\in\pa\om$. Now, since $\pa\om$ is of class $C^{m,\al}$, up to a change of coordinates, we can assume that $x=(0,\dots,0,x_N)$ and find a small neighborhood $V$ of $x$, an $N-1$ dimensional open ball $B$ centered at the origin of $\RR^{N-1}$ and a function $f\in C^{m,\al}(\RR^{N-1})$ such that 
\begin{equation*}
\pa\om\cap V=\setbld{(\hat x, x_N)\in \RR^{N-1}\times \RR}{\hat x\in B, \; x_N=f(\hat x)}.
\end{equation*}
For arbitrary $\hat x\in \RR^{N-1}$ and $x_N\in \RR$, set $\pi(\hat x, x_N):=\hat x$ and $\psi(\hat x):= (\hat x, f(\hat x))$ (notice that $\pi\circ\psi=\id$ and $\restr{\psi\circ\pi}{\pa\om\cap V}=\restr{\id}{\pa\om\cap V}$).
Now, consider the following function: 
\begin{eqnarray*}
F: & B\times \cU \longrightarrow & B\times \cU\\
& (y,\theta) \mapsto &\left(\pi\circ G(\theta)\circ \psi(y), \theta \right).
\end{eqnarray*}
We will show that $F$ is locally invertible in a neighborhood of $(0,0)\in B\times\cU$ and this will lead to a contradiction. 
To this end, first notice that $F$ is well defined and Fr\'echet differentiable in a neighborhood of $(0,0)\in B\times\cU$. Moreover, its (total) Fr\'echet derivative at $(0,0)$ in the direction $(y,\theta)\in \RR^{N-1}\times \Theta^{m+1,\al}$ is given by
\begin{equation*}
F'(0,0)[y,\theta]= \pa_y F(0,0)[y]+ \pa_\theta F(0,0)[\theta]= (y,0)+(-\pi\circ\theta_\tau(x),\theta).
\end{equation*}
We remark that the above computation easily follows from the equalities below
\begin{equation*}
G(0)=\id, \quad \pi'(0)=\pi, \quad \psi(0)=x, \quad G'(0)[\theta]=-\theta_\tau:=-\theta+ (\theta\cdot n_\om)n_\om,
\end{equation*}
and we refer to \cite[(3.11)]{structure} for the proof of the last equality. It follows that $F'(0,0)$ has a continuous inverse given by 
\begin{equation*}
(y,\theta)\mapsto \left( y+\pi\circ \theta_\tau(x), \theta \right).
\end{equation*}
Therefore, by the inverse function theorem (also known as the ``local inversion theorem", see \cite[Theorem 1.2]{AP1983}), there exists an open neighborhood $\cV$ of $(0,0)\in B\times \cU$ such that the restriction $\restr{F}{\cV}: \cV\to F(\cV)$ is invertible and its inverse is of class $C^1$. In particular, the set $F(\cV)$ is an open neighborhood of $(0,0)=F(0,0)$. This implies the existence of some $\ve>0$ such that $(\hat x,\theta)\in F(\cV)$ whenever $|\hat x|+\norm \theta <\ve$.By construction, if $|\hat x|+\norm \theta <\ve$, then there exists a point $y\in B$ such that $(y,\theta)\in \cV$ and $F(y,\theta)=(\hat x,\theta)$. In particular, for $n\in\NN$ sufficiently large, we can take $\hat x=\pi(x_n)$ and $\theta=\theta_n$ in the above. Thus, there exists a point $y_n\in B$ such that 
\begin{equation*}
F(y_n,\theta_n)=\left( \pi(x_n),\theta_n\right) \implies   \pi\circ G(\theta_n)\circ \psi(y_n)=\pi(x_n).
\end{equation*}
Now, applying $\psi$ on both sides of the second equality above yields $G(\theta_n)\circ\psi(y_n)=x_n$. Finally, since $\psi(y_n)\in \pa\om$ by construction, this contradicts the assumption \eqref{no surj}. In other words, we have shown that the map $G(\theta_n)$ is surjective for $n$ large enough. By analogous reasoning, one can show that $G(\theta_n)$ is injective for $n$ large enough. 
\end{proof}
As a consequence of Lemma \ref{reparametrization} and Proposition \ref{G is bij}, we can state the following result.  
\begin{corollary}[Small rotations can be represented as Hadamard perturbations]\label{corol rotations}
Let $\om\in\cC^{m+1,\al}$ and let $\{R_\vartheta\}_{\vartheta\in\RR}$ be a one-parameter group of rotations. Then, there exist a threshold $\varepsilon>0$ and a $C^1$ map $\Psi:(-\varepsilon,\varepsilon) \to C^{m,\al}(\pa\om)$ such that, for all $\vartheta\in (-\varepsilon,\varepsilon)$, we have
\begin{equation*}
    R_\vartheta(\om) = \om_{\Psi(\vartheta)}.
\end{equation*}
\end{corollary}

Given a shape functional $J:\cC^{m,\al}\to\RR$ and an open set $\om\in\cC^{m+1,\al}$, set
\begin{equation*}
I(\te):=J\left((\id+\te)(\om)\right), \quad j_\om(\xi):=J(\om_\xi)    
\end{equation*}
for all $\te\in\Theta^{m+1,\al}$, $\xi\in C^{m+1,\al}(\pa\om)$ small enough. Then, for small $\te\in\Theta^{m+1,\al}$, Lemma \ref{reparametrization} yields
\begin{equation*}
I(\te):= J\left((\id+\te) (\om) \right) = 
J\left(\om_{\Psi(\te)}\right)=j_\om(\Psi(\te)),
\end{equation*}
where in the second equality we employed the fact that $    (\id+\te)(\pa\om)=(\id+\te)\circ G(\te)(\pa\om)
$ because $G(\te)$ is a bijection on $\pa\om$. 
As a result, if $j_\om$ is Fr\'echet differentiable at $0\in C^{m,\al}(\pa\om)$, then $I$ is Fr\'echet differentiable at $0\in \Theta^{m+1,\al}$. Now, differentiating the above at $\te=0$ yields
\begin{equation*}
I'(0)[\te] = j_\om'(0)\left[\Psi'(0)[\te]\right]=j_\om'(0)\left[ \restr{\te}{\pa\Om}\cdot n\right], \quad \text{for all }\te\in\Theta^{m+1,\al}.
\end{equation*}
The above identity can be interpreted as follows: the result of shape differentiation with respect to domain perturbations of the form $(\id+\te)(\om)$ only depends on the normal component of the restriction of the perturbation field $\te$ on $\pa\om$; moreover, it does so in a linear fashion. This is the so-called structure theorem for (first-order) shape derivatives (\cite[Theorem 2.1, (i)]{structure}). 

In the following sections we will also make use of the following corollary, simply obtained by the chain rule:
\begin{corollary}\label{corol structure}
Let $t\mapsto \Phi(t)$ be a $C^1$ map from a small open interval $(-\ve,\ve)$ to $\Theta^{m+1,\al}$. If $\Phi(t)$ admits the following Taylor expansion in $\Theta^{m+1,\al}$ for some $\varphi\in \Theta^{m+1,\al}$
\begin{equation*}
\Phi(t)=\id+t\varphi+o(t) \quad \text{as }t\to 0,
\end{equation*}
then 
\begin{equation*}
\restr{\frac{d}{dt} J\Big(\Phi(t)(\om)\Big)}{t=0}=j_\om'(0)\left[\restr{\varphi}{\pa\om}\cdot n_\om\right]. 
\end{equation*}
\end{corollary}

\section{The perturbation theory for the critical points of a shape functional}\label{section nondegeneracy}

Let $m\in\NN$ and $\al\in[0,1]$. Moreover, let $\La$ be (an open subset of) a Banach space (that we will think of as the space of ``parameters"). First of all, we will consider a \emph{parametrized} shape functional $\cJ:\cC^{m,\al}\times \La\to \RR$. Notice that the case of a ``non-parametrized" shape functional $J:\cC^{m,\al}\to\RR$ is trivially included by setting $\cJ(\om,\la):=J(\om)$ for all $\la\in\La$.

Now, 
suppose that, for $\om\in\cC^{m+1,\al}$, the map
\begin{equation}\label{J_om def}
\begin{aligned}
j_\om:\quad & C^{m,\al}(\pa\om)\times \La\to \RR, \\
\ & (\xi,\la)\mapsto J(\om_\xi,\la)
\end{aligned}
\end{equation}
is well-defined and Fr\'echet differentiable in a neighborhood $X\times \La$ of $(0,0)\in C^{m,\al}(\pa\om)\times\La$. 
By the structure theorem for first order shape derivatives (see \cite[Theorem 2.27]{SZ1992}, \cite[Theorem 2.1, (i)]{structure} for a general version of this result in the case of shape derivatives with respect to perturbations of the type \eqref{interior perturbation}), for any $\om\in\cC^{m+1,\al}$, there exists a distribution $T_{\om,\la}$ of order at most $m+1$ concentrated at $\pa\om$ such that the shape derivative of $J(\cdottone,\la)$ at $\Om$ in the direction $\xi$ can be expressed as $\langle T_{\om,\la},\xi\rangle$.
In what follows we will assume that for all $\om\in \cC^{m+1,\al}$ and $\la\in\La$, the distribution $T_{\om,\la}$ can be expressed as a function
\begin{equation}\label{g_om}
g_\om(\la)\in L^1(\pa\om).
\end{equation}
In other words, we assume that the partial Fr\'echet derivative of $j_\om$ with respect to the first variable is given by \begin{equation}\label{first derivative}
\pa_x j_\om(0,\la)[\xi]=\int_{\pa\om} g_\om(\la)\xi \quad \text{for all }\xi \in C^{m,\al}(\pa\om)\text{ and }\la\in\La.
\end{equation}
We remark that, in most applications, the function $g_\om(\la)$ will be far more regular.

We say that an open set $\om\in\cC^{m+1,\al}$ is a critical shape for $\cJ(\cdottone,\la)$ if 
\begin{equation}
\pa_x j_\om(0,\la)[\xi]= 0 \quad \text{for all }\xi\in\cC^{m,\al}(\pa\om),
\end{equation}
that is, $g_\om(\la)=0$.
Let $\Om\in\cC^{m+2,\al}$ be a critical shape for $\cJ(\cdottone,\la)$ and let $n$ denote the outward unit normal vector at $\pa\Om$. Moreover, suppose that for, some Banach space $Y\subset L^1(\pa\Om)$, the mapping
\begin{equation}\label{h def}
\begin{aligned}
h:\quad & C^{m+1,\al}(\pa\Om)\times \La\to Y, \\
\ & (\xi,\la)\mapsto g_{\Om_\xi}(\la)\circ \pp{\id+ \xi n} \end{aligned}
\end{equation}
is well-defined and Fr\'echet differentiable in a neighborhood $X\times \La$ of $(0,0)\in C^{m+1,\al}(\pa\Om)\times \La$.
Then, by composition, 
the mapping $j_\Om:X\times\La\to \RR$
is twice Fr\'echet differentiable at $\xi=0$ and there exists a bounded linear operator $Q: C^{m+1,\al}(\pa\Om) \to Y$ such that
\begin{equation}\label{Q}
\pa_{xx}^2\ j_\Om(0,0)[\xi,\eta]= \int_{\pa\Om} Q(\xi) \eta \quad \text{for all }\xi,\eta\in C^{m+1,\al}(\pa\Om).
\end{equation}

Now, employing the function $g_\Om(0)$ and the operator $Q$ constructed above, we can give the following definition of nondegenerate critical shapes. 
\begin{definition}[Nondegenerate critical shape]\label{def nondeg1}
Employing formulas \eqref{first derivative} and \eqref{Q}, we say that $\Om\in\cC^{m+2,\al}$ is a \emph{nondegenerate} critical shape for $\cJ:\cC^{m,\al}\times \La\to\RR$ at $\la=0$ if the following two conditions hold:
\begin{enumerate}[label=(\roman*)]
\item $\pa_x j_\Om(0,0)[\cdottone]\equiv 0$. In other words, $g_\Om(0)\equiv 0$ on $\pa\Om$ (\emph{criticality});
\item the mapping $Q:C^{m+1,\al}(\pa\Om)\to Y$ is a bijection (\emph{nondegeneracy}).
\end{enumerate}
\end{definition}

The following result (\cite[Theorem I]{cava nondegenerate}) shows the link between nondegeneracy and the existence of a uniquely defined parametrized family of critical shapes for small $\la$. 

\begin{mainlem}\label{lem I}
Suppose that $\Om\in\cC^{m+2,\al}$ is a nondegenerate critical shape (according to Definition \ref{def nondeg1}) for a shape functional $\cJ:\cC^{m,\al}\times\La\to\RR$ at $\la=0$. Then, there exist open neighborhoods $X'$ of $0\in C^{m+1,\al}(\pa\Om)$ and $\La'$ of $0\in \La$, and a Fr\'echet differentiable map $\widetilde\xi: \La'\to X'$ such that the set $\Om_{\widetilde\xi(\la)}$ is a critical shape for the shape functional $\cJ(\cdottone, \la)$. Moreover, for $(\xi,\la)\in X'\times \La'$, the set $\Om_{\xi}$ is a critical shape for $\cJ(\cdottone,\la)$ if and only if $\xi=\widetilde\xi(\la)$.
\end{mainlem}

In what follows we will discuss how to modify Definition \ref{def nondeg1} and Lemma \ref{lem I} to deal with parametrized shape functionals of the form $\cJ:\cC^{m,\al}_\star\times \La\to\RR$. 
For any given $\om\in \cC^{m+1,\al}_\star$, and $\xi\in\C^{m+1,\al}(\pa\om)$, set
\begin{equation}\label{bar}
b_\om(\xi):= {\rm Bar}(\om_\xi)= \frac{\int_{\om_\xi} x\, dx}{|\om_\xi|}. 
\end{equation}
It is well known that the function $b_\om$ is well-defined and Fr\'echet differentiable in a neighborhood of $0\in C^{m+1,\al}(\pa\om)$.  
For any given $\om\in\cC^{m+1,\al}_\star$, we define the space of perturbations $C^{m+1,\al}_\star(\pa\om)$ that do not alter the barycenter of $\om$ at first order:
\begin{equation*}
C^{m+1,\al}_\star(\pa\om):= 
\ker b_\om'(0) =\setbld{\xi\in C^{m+1,\al}(\pa\om) }{b_\om'(0)[\xi]=0}. 
\end{equation*}
We stress that the elements of $C^{m+1,\al}_\star(\pa\om)$ only satisfy the barycenter constraint at \emph{first order}. In other words, for $\xi\in C^{m+1,\al}_\star(\pa\om)$, the perturbed set $\om_\xi$ does not necessarily belong to $\cC^{m,\al}_\star$. Therefore, in order to define the shape derivative of a shape functional $J:\cC^{m,\al}_\star\times \La\to\RR$, we will first need to project $\om_\xi$ back to the constraint space $\cC^{m,\al}_\star$ by  
\begin{equation*}
\om\mapsto  
\om-{\rm Bar}(\om) 
\end{equation*}
By the construction above, any given shape functional $\cJ:\cC^{m,\al}_\star\times \La\to\RR$ admits (and is uniquely identified by) a translation-invariant extension (here, denoted by the same letter) $\cJ: \cC^{m,\al}\times \La\to\RR$, defined as 
\begin{equation}\label{extension}
\cJ(\om,\la):=\cJ(\om-{\rm Bar}(\om),\la) \quad\tfor \om\in\cC^{m,\al}\setminus\cC^{m,\al}_\star. 
\end{equation}
In particular, for $\om\in\cC^{m+1,\al}$, we can define the map $j_\om: C^{m+1,\al}(\pa\om)\times \La\to \RR$ as in \eqref{J_om def}. Analogously, assume that $j_\om$ is well-defined and Fr\'echet differentiable in a neighborhood of $(0,0)\in C^{m+1,\al}(\pa\om)\times\La$ and define the map $g_\om:\La\to L^2(\pa\om)$ as in \eqref{g_om}--\eqref{first derivative}. Finally, suppose that the mapping $h$ defined in \eqref{h def} is well-defined and Fr\'echet differentiable and let $Q:C^{m+1,\al}(\pa\Om)\to Y$  denote the bounded linear operator defined as in \eqref{Q}. 

We are now ready to state the definition of nondegenerate critical shape for shape functionals of the form $\cJ:\cC^{m,\al}_\star\times \La\to\RR$ as follows. 

\begin{definition}[Nondegenerate critical shape for restricted shape functionals]\label{def nondeg2}
Let the notation be as above. We say that $\Om\in\cC^{m+2,\al}_\star$ is a \emph{nondegenerate} critical shape for 
$\cJ:\cC^{m,\al}_\star\times \La\to\RR$ at $\la=0$ if the following two conditions hold:
\begin{enumerate}[label=(\roman*)]
\item $\pa_x j_\Om(0,0)[\xi]= 0$ for all $\xi\in C^{m+1,\al}_\star(\pa\Om)$. In other words, the function $g_\Om(0)$ belongs to the orthogonal complement of $C^{m+1,\al}_\star(\pa\Om)$ in $L^2(\pa\Om)$ (\emph{criticality});
\item the restriction $\restr{ Q}{C^{m+1,\al}_\star(\pa\Om)}:C^{m+1,\al}_\star(\pa\Om)\to Y$ is a bijection (\emph{nondegeneracy}).
\end{enumerate}
\end{definition}


The following result shows the link between nondegeneracy and existence of a uniquely defined parametrized family of critical shapes for small $\la$. 

\begin{mainlem}\label{lem II}
Suppose that $\Om\in\cC^{m+2,\al}_\star$ is a nondegenerate critical shape (according to Definition \ref{def nondeg2}) for a shape functional $\cJ:\cC^{m,\al}_\star\times\La\to\RR$ at $\la=0$. Then, there exist open neighborhoods $X'$ of $0\in C^{m+1,\al}(\pa\Om)$ and $\La'$ of $0\in \La$ and a Fr\'echet differentiable map $\widetilde\xi: \La'\to X'$ such that the set $\Om_{\widetilde\xi(\la)}$ belongs to $\cC^{m+1,\al}_\star$ and is a critical shape for the shape functional $\cJ(\cdottone, \la)$. Moreover, for any pair $(\xi,\la)\in X'\times \La'$ such that ${\rm Bar}(\Om_{\xi})=0$ (that is, $\Om_{\xi}\in\cC^{m+1,\al}_\star$), the set $\Om_{\xi}$ is a critical shape for $\cJ(\cdottone,\la)$ if and only if $\xi=\widetilde\xi(\la)$.
\end{mainlem}
\begin{remark}
The definition of \emph{nondegeneracy} of a critical point given in Definitions \ref{def nondeg1}--\ref{def nondeg2} can be thought of as a generalization of that used by Smale, Palais and Tromba in \cite{smale, Palais 69, tromba}. Moreover, by considering $Y$ to be a (possibly) larger space than $C^{m+1,\al}(\pa\Om)$ (respectively $C^{m+1,\al}_\star(\pa\Om)$), we can take into account the ``derivative loss" that usually occurs when dealing with shape derivatives.
\end{remark}

\begin{remark}
If equations
\begin{equation*}
Q(\xi)=\eta,\quad \restr{ Q}{C^{m+1,\al}_\star(\pa\Om)}(\xi)=\eta
\end{equation*}
for $\xi,\eta$ in the appropriate Banach spaces, satisfy the Fredholm alternative (as it is often the case when one considers shape functionals that depend on the solutions of elliptic boundary value problems), then the nondegeneracy assumptions $(ii)$ of Definitions \ref{def nondeg1}-\ref{def nondeg2} simply become $\ker Q=\{0\}$ and $\ker \restr{ Q}{C^{m+1,\al}_\star(\pa\Om)}=\{0\}$ respectively.
\end{remark}

In practice, the following result will be a useful tool to show the nondegeneracy of a critical shape $\Om$. 

\begin{proposition}
Let X denote either $C^{m+1,\al}(\pa\Om)$ or $C_\star^{m+1,\al}(\pa\Om)$. Suppose that $X$ is compactly embedded in $Y$ via the inclusion mapping $\iota: X\hookrightarrow Y$ and that for some $\mu\in\RR$ the map $Q+\mu\iota: X\to Y$ is a bijection. Then the following hold.
\begin{enumerate}[label=(\roman*)]
    \item Let $\cS\subset\RR$ denote the set of all real numbers $\la$ such that the following equation admits a nontrivial solution $u\ne 0$ in $X$:
    \begin{equation}\label{eigenequation}
        Qu=\la u.
    \end{equation}
    Then $\cS$ is a countable set $\{\la_n\}_n$, $\displaystyle \lim_{n\to\infty}|\la_n|=\infty$, and the space of solutions of \eqref{eigenequation} is finite-dimensional for all $\la=\la_n\in\cS$. 
    \item $\Om$ is non-degenerate (according to either Definition \ref{def nondeg1} or \ref{def nondeg2}) if and only if $0\notin \cS$.
\end{enumerate}
\end{proposition}
\begin{proof}
Let us first show $(i)$. Consider the operator $K:=(Q+\mu\iota)^{-1}:Y\to X\hookrightarrow Y$. By construction, $K$ is a compact operator from $Y$ into itself. Thus, by the spectral theorem \cite[Theorem 5.5]{GT}, $K$ possesses a countable set of eigenvalues $\{\La_n\}_{n\in\NN}$ with $0$ as an accumulation point, and the corresponding eigenspaces are finite-dimensional. In other words, there exists a sequence of nonzero functions $\{u_n\}_{n\in\NN}$ in $X\setminus \{0\}$ such that $Ku_n=\La_n u_n$. Recalling that $K=(Q+\mu\iota)^{-1}$ and rearranging the terms in this equation yield the desired
\begin{equation*}
Q u_n = \underbrace{(1/\La_n-\mu)}_{=: \la_n} u_n,
\end{equation*}
which holds for all $\La_n\ne 0$.

Let us now show $(ii)$. To this end, we will first show that the equation $(\id-\mu K)u=0$ has a nonzero solution $u\in Y\setminus\{0\}$ if and only if $0\in\cS$. Indeed, for $u\in Y\setminus\{0\}$ we have:
\begin{equation*}
u=\mu K u \iff (\mu \iota + Q)u=\mu u \iff Qu =0 \iff 0\in \cS. 
\end{equation*}
The above can be rephrased as follows: $0\notin \cS$ if and only if $u=0$ is the only solution to $(\id-\mu K)u=0$ in $Y$. 
Thus, by the Fredholm alternative (Riesz--Schauder theory) \cite[Theorem 5.3]{GT}, $0\notin \cS$ if and only if $\id-\mu K $ admits a continuous inverse $T:Y\to Y$. Now, for $u,y\in Y$ we have
\begin{equation*}
Qu=y \iff \mu u + Qu = \mu u + y \iff
u= K(\mu  u + y) \iff
(\id - \mu K )u= Ky \iff u= TK y.
\end{equation*}
In other words, $0\notin \cS$ if and only if $Q: X\to Y$ admits a continuous inverse, which is given by $Q^{-1}=TK$. 
\end{proof}

\section{Proof of Lemma \ref{lem II}}\label{sec 5}

Lemmas \ref{lem I}-\ref{lem II} are one of the main ingredients in the proof of Theorems \ref{mainthm I}-\ref{mainthm II}. The result of Lemma \ref{lem I} was originally stated and proved in \cite[Theorem I]{cava nondegenerate}, so, here, we will just give a proof of Lemma \ref{lem II}.

Before proving Lemma \ref{lem II}, some preliminary work is needed. In what follows, let us consider a fixed open set $\Om\in\cC^{m+2,\al}_\star$ and a parametrized shape functional $\cJ:\cC_\star^{m,\al}\times\La\to\RR$ (when no confusion arises, $\cJ$ will also denote its translation-invariant extension to $\cC^{m,\al}\times\La$). For simplicity, let $\be:C^{m+1,\al}(\pa\Om)\to\rn$ denote the Fr\'echet derivative at $\xi=0$ of the barycenter function $\xi\mapsto b_\Om(\xi)$ defined in \eqref{bar}, that is
\begin{equation*}
    \be(\xi):= b_\Om'(0)[\xi]= \frac{1}{|\Om|} \int_{\pa\Om} x \xi(x)\ dS_x, \quad \text{for all }\xi\in C^{m+1,\al}(\pa\Om).
\end{equation*}
Making use of $\be$, we define the following projection mapping:
\begin{equation*}
\pi_\star(\xi):=\xi- \be(\xi)\cdot n.
\end{equation*}
We have the following result concerning $\pi_\star$:
\begin{proposition}\label{projection}
The mapping $\pi_\star: C^{m+1,\al}(\pa\Om) \to C^{m+1,\al}_\star(\pa\Om)$ is a bounded linear projection. 
\end{proposition}
\begin{proof}
Linearity and boundedness (that is, continuity) immediately follow from the definition of $\pi_\star$. It just remains to show that, for all $\xi\in C^{m+1,\al}(\pa\Om)$, $\be\left(\pi_\star(\xi)\right)=0$. To this end, take an arbitrary point $y_0\in\rn$. Now, for $t\in\RR$, we have
\begin{equation*}
{\rm Bar}(\Om+ty_0)=ty_0. 
\end{equation*}
Differentiating both members with respect to $t$ at $t=0$ (using Corollary \ref{corol structure} for the left hand side) yields
\begin{equation}\label{identity}
\be(y_0\cdot n)=y_0.
\end{equation}
Finally, by the linearity of $\be$ and \eqref{identity} with $y_0=\be(\xi)$, we conclude that
\begin{equation*}
\be\left( \pi_\star(\xi) \right)= \be\left( \xi - y_0\cdot n\right)=
\be\left( \xi \right) - \be(y_0\cdot n)= \be(\xi)-y_0=0,
\end{equation*}
which is what we wanted to show.
\end{proof}

\begin{proposition}
Let $\Om\in\cC^{m+2,\al}_\star$ be a critical shape (according to Definition \ref{def nondeg2}) for the parametrized shape functional $\cJ:\cC^{m,\al}_\star\times\La\to\RR$. Then $g_\Om(0)\equiv0$ on $\pa\Om$.
\end{proposition}
\begin{proof}
Let $\xi\in C^{m+1,\al}(\pa\Om)$. We will compute the Fr\'echet derivative $\pa_x j(0,0)[\xi]$ as a G\^ateaux derivative. We have
\begin{equation*}
\int_{\pa\Om}g_\Om(0)\ \xi = \pa_x j(0,0)[\xi]=\restr{\frac{d}{dt}}{t=0} j(t\xi,0)= \restr{\frac{d}{dt}}{t=0} \cJ(\Om_{t\xi},0)=\restr{\frac{d}{dt}}{t=0} \cJ\left(\widetilde\Om_{t\xi},0 \right),
\end{equation*}
where we have set $\widetilde\Om_{t\xi}:=\Om_{t\xi}-{\rm Bar}(\Om_{t\xi})$. Notice that $\pa\widetilde\Om_{t\xi}=\Phi(t)(\pa\Om)$, where $\Phi(t):\rn\to\rn$ is a smooth map such that 
\begin{equation*}
\restr{\Phi(t)}{\pa\Om}= \id+t\xi n -t\be(\xi) + o(t) \quad \tas t\to 0.
\end{equation*}
Furthermore, notice that the normal component of the first order term of the perturbation above is given by $\varphi:=\xi-\be(\xi) \cdot n=\pi_\star(\xi)$. In particular, $\varphi\in C^{m+1,\al}_\star(\pa\Om)$ by Proposition \ref{projection}. Thus, Corollary \ref{corol structure} implies
\begin{equation*}
\restr{\frac{d}{dt}}{t=0} \cJ\left(\widetilde\Om_{t\xi},0 \right)= j_\Om'(0,0)[\pi_\star(\xi)]=0,
\end{equation*}
where, in the last equality, we made use of the fact that $j_\Om'(0,0)[\cdottone]\equiv 0$ in $C^{m+1,\al}_\star(\pa\Om)$ by hypothesis (criticality). Since $\xi\in C^{m+1,\al}(\pa\Om)$ is arbitrary, the claim follows.
\end{proof}

\begin{proposition}\label{Q li seme}
For all $\xi\in C^{m+1,\al}(\pa\Om)$, we have 
$Q(\xi)=\pa_x h(0,0)[\xi]$.
\end{proposition}
\begin{proof}
By hypothesis, the mapping $(\xi,\la)\mapsto j_\Om(\xi,\la)$ is Fr\'echet differentiable in a neighborhood of $(0,0)\in C^{m,\al}(\pa\Om)\times \La$. Computing the partial derivative with respect to the first variable at the point $(\xi,\la)\in C^{m+1,\al}(\pa\Om)\times\La$ in the direction $\eta\in C^{m+1,\al}(\pa\Om)$ with Corollary \ref{corol structure} at hand yields
\begin{equation*}
\pa_x j (\xi,\la)[\eta]= \restr{\frac{d}{dt}}{t=0} j(\xi+t\eta,\la)= 
 \int_{\pa\Om_{\xi}}  g_{\Om_{\xi}}(\la) \ \pp{\pp{\eta n}\circ \left(\id+\xi n\right)^{-1} \cdot n_{\xi}},
\end{equation*}
where $n_{\xi}$ denotes the outward unit normal vector to $\pa\Om_{\xi}$.
By a change of variables, the expression above can be rewritten as
\begin{equation}\label{pax j}
\pa_x j(\xi,\la)[\eta]= \int_{\pa\Om} h(\xi,\la)\ m(\xi)\ \eta,
\end{equation}
where
\begin{equation}\label{h&m}
h(\xi,\la):= g_{\Om_{\xi}}(\la) \circ \pp{\id+\xi n}, \quad m(\xi):= J_\tau(\xi)\ n_{\xi}\circ\pp{\id+\xi n}\cdot n,
\end{equation}
and $J_\tau(\xi)$ denotes the tangential Jacobian associated to the map $\id+ \xi n$ (see \cite[Definition 5.4.2 and Proposition 5.4.3]{HP2018}).
It is known (see \cite[Proposition 5.4.14 and Lemma 5.4.15]{HP2018}) that both the normal vector and the tangential Jacobian are Fr\'echet differentiable with respect to perturbations of class $C^1$. Moreover, by hypothesis, we know that also the mapping 
\begin{equation*}
(\xi,\la)\mapsto g_{\Om_{\xi}}(\la)\circ\pp{\id+\xi n}:=h(\xi,\la)\in Y    
\end{equation*} is well-defined and Fr\'echet differentiable in a neighborhood of $(0,0)\in C^{m+1,\al}(\pa\Om)\times \La$. 
By composition, both $h(\cdottone,\cdottone)$ and $m(\cdottone,\cdottone)$ are Fr\'echet differentiable in a neighborhood of $(0,0)\in X\times \La$. In particular, this implies that, for fixed $\eta$, also $\pa_x j(\cdottone,\cdottone)[\eta]$ is Fr\'echet differentiable in a neighborhood of $(0,0)\in C^{m+1,\al}(\pa\Om)\times \La$.
Differentiating now \eqref{pax j} with respect to the first variable one more time at the point $(0,0)$ yields
\begin{equation*}
\pa_{xx}^2\ j(0,0)[\xi,\eta]= \int_{\pa\Om} \pp{\pa_x h (0,0)[\xi]\ m(0,0)+h(0,0)\ \pa_x m(0,0)[\xi]}\ \eta = \int_{\pa\Om} \underbrace{\pa_x h(0,0)[\xi]}_{:=Q(\xi)}\ \eta,
\end{equation*}
where we have made use of the following identities:
\begin{equation*}
h(0,0)=g_\Om(0)=0,\quad m(0,0)= 1.
\end{equation*}
In other words, the bounded linear operator $Q$ defined in \eqref{Q} is nothing but $\pa_x h(0,0)$, as claimed.
\end{proof}

\begin{lemma}\label{Q(x0n)}
For all $x_0\in\rn$, the following holds:
\begin{equation*}
    Q(x_0\cdot n)=0.
\end{equation*}
\end{lemma}
\begin{proof}
Fix $\xi_0\in\rn$, $\eta\in C^{m+1,\al}(\pa\Om)$ and for small $t$, let $\Psi(t)\in C^{m+1,\al}(\pa\Om)$ and $G(t)\in C^{m+1,\al}(\pa\Om,\pa\Om) $ be the functions given by Lemma \ref{reparametrization} that satisfy
\begin{equation}\label{identity 2}
(\id+ tx_0)\circ G(t) = \id+ \Psi(t)n \quad \ton \pa\Om, \quad \Psi'(0)=x_0\cdot n
\end{equation}
Since $\Om_\eta+tx_0=(\Om+tx_0)_{\eta\circ(\id+tx_0)}$, we have the following expansion as $\eta\to 0$ in $C^{m+1,\al}(\pa\Om)$:
\begin{equation*}
\begin{aligned}
\cJ(\Om_\eta+tx_0,0)=\cJ\left( (\Om+tx_0)_{\eta\circ(\id+tx_0)}, 0\right)= \cJ(\Om+tx_0,0)+ \int_{\pa\Om+tx_0} g_{\Om+tx_0}(0) \ \eta\circ(\id-tx_0) + o(\eta) \\
= \cJ(\Om,0) + \int_{\pa\Om_{\Psi(t)}} g_{\Om_{\Psi(t)}}(0)\ \eta\circ(\id-tx_0) + o(\eta) = 
\cJ(\Om,0) + \int_{\pa\Om} h(\Psi(t),0)\ \eta \circ G(t) J_\tau(t) + o(\eta),
\end{aligned}
\end{equation*}
where we made use of \eqref{identity 2} in the last equality and let $J_\tau (t)$ denote the tangential Jacobian associated to the change of variables in the surface integral. A further change of variables, this time with respect to the mapping $G^{-1}(t):\pa\Om\to\pa\Om$, yields
\begin{equation}\label{one way}
\cJ(\Om_\eta+tx_0,0)= \cJ(\Om,0) + \int_{\pa\Om} h(\Psi(t),0)\circ G^{-1}(t) \ \eta\ \widetilde J_\tau(t) + o(\eta),
\end{equation}
where $\widetilde J_\tau(t)$ is the combined tangential Jacobian associated to the two previous changes of variables.

On the other hand, since $\cJ(\cdottone,0)$ is translation-invariant by construction, the following expansion as $\eta\to 0$ holds true as well
\begin{equation}\label{another way}
\cJ(\Om_\eta+tx_0,0)= \cJ(\Om_\eta,0)= \cJ(\Om)+ \int_{\pa\Om}g_\Om(0)\ \eta +o(\eta).
\end{equation}
Comparing the first order terms in \eqref{one way}-\eqref{another way} and recalling that the function $\eta\in C^{m+1,\al}(\pa\Om)$ was arbitrary yield
\begin{equation*}
h(\Psi(t),0)\circ G^{-1}(t) \widetilde J_\tau(t) = g_\Om(0)=0 \quad \ton\pa\Om.
\end{equation*}
Therefore, since $\widetilde J_\tau(t)$ is strictly positive on $\pa\Om$ and $G^{-1}(t)$ is a bijection, the identity above implies 
\begin{equation*}
    h\left(\Psi(t),0\right)\equiv 0\quad \ton\pa\Om, \quad \text{for $|t|$ small}. 
\end{equation*}
Finally, differentiating the above by $t$ at $t=0$, Proposition \ref{Q li seme} and the second identity in \eqref{identity 2} imply the desired identity 
\begin{equation*}
0=\pa_x h(0,0)\left[\Psi'(0)\right]=Q(x_0\cdot n). 
\end{equation*}
\end{proof}

In what follows we will give a proof of Lemma \ref{lem II}. One of the main ingredients is the following version of the implicit function theorem for Banach spaces (see \cite[Theorem 2.3, page 38]{AP1983} for a proof). 
\begin{thm}[Implicit function theorem]\label{ift}
Let $\cH\in\cC^k(X\times \La,Z)$, $k\ge1$, where $Z$ is a Banach space and $X$ (resp. $\La$) is an open set of a Banach space $\widetilde X$ (resp. $\widetilde \La$). Suppose that  $\cH(x^*,\la^*)=0$ and that the partial derivative $\pa_x\cH(x^*,\la^*)$ is a bounded invertible linear transformation from $X$ to $Z$. 

Then, there exist neighborhoods $\La'$ of $\la^*$ in $\widetilde \La$ and $X'$ of $x^*$ in $\widetilde X$, and a map $\xi\in\cC^k(\La',X')$ such that the following hold:
\begin{enumerate}[label=(\roman*)]
\item $\cH(\xi(\la),\la)=0$ for all $\la\in\La$,
\item If $\Psi(x,\la)=0$ for some $(x,\la)\in X'\times \La'$, then $x=\xi(\la)$,
\item $\xi'(\la)=-\left(\pa_x \cH(p) \right)^{-1}\circ \pa_\la \cH(p)$, where $p=(\xi(\la),\la)$ and $\la\in\La'$.
\end{enumerate}
\end{thm}

\begin{proof}[Proof of Lemma \ref{lem II}]
The proof follows by applying the implicit function theorem (Theorem \ref{ift}) to the following functional defined in a small enough neighborhood $X\times \La$ of $(0,0)\in C^{m+1,\al}(\pa\Om)\times \La$:
\begin{eqnarray*}
\cH: X\times \La \longrightarrow Y\times \rn,\\
  (\xi,\la)\mapsto\left(h(\xi,\la), {\rm Bar}(\Om_\xi) \right).
\end{eqnarray*}
The functional $\cH$ verifies $\cH(0,0)=(0,0)$ by construction and is clearly Fr\'echet differentiable in a neighborhood of $(0,0)\in X\times \La$ because both its components are. Moreover, its partial Fr\'echet derivative with respect to the first variable is given by the map $\xi\mapsto \left(Q(\xi),\be(\xi)\right)$. It now suffices to show that this map is a bijection. To this end, take an arbitrary pair $(\eta,y_0)\in Y\times \rn$ and consider the equation 
\begin{equation}\label{equation}
Q(\xi)=\eta, \quad \be(\xi)=y_0.    
\end{equation}
The function $\xi$ can be decomposed as the sum $\xi_\star+\xi_\perp$, where $\xi_\star:=\pi_\star(\xi)\in C^{m+1,\al}_\star(\pa\Om)$ and $\xi_\perp:=\be(\xi)\cdot n$. 
By the linearity of $Q$ and Lemma \ref{Q(x0n)} with $x_0=\be(\xi)$, equation \eqref{equation} can be rewritten as 
\begin{equation*}
Q(\xi_\star)=\eta, \quad \be(\xi)=y_0.    
\end{equation*}
Thus, since the restriction $\restr{Q}{C^{m+1,\al}_\star(\pa\Om)}\to Y$ is a bijection by hypothesis, the equation above is uniquely solvable and its solution is given by 
\begin{equation*}
\xi=\xi_\star+\xi_\perp = Q^{-1}(\eta)+ y_0\cdot n.
\end{equation*}
The claim now follows from the implicit function theorem.
\end{proof}

\section{Putting all pieces together: proof of Theorems \ref{mainthm I}--\ref{mainthm II}}\label{sec 6}
Let $\Om\in\cC^{m+2,\al}$ (respectively, $\cC^{m+2,\al}_\star$) be a nondegenerate critical point for a shape functional $J:\cC^{m,\al}\to\RR$ (respectively, $J:\cC^{m,\al}_\star\to\RR$). Moreover, suppose that $J$ is invariant with respect to some one-parameter group of rotations $\{R_\vartheta\}_{\vartheta\in\RR}$. We can now define a parametrized shape functional by setting
\begin{eqnarray*}
\cJ(\om,\vartheta):= J(R_\vartheta(\om))    
\end{eqnarray*}
for all $\vartheta\in\RR$ and $\om\in\cC^{m,\al}$ (respectively $\cC^{m,\al}_\star$). Furthermore, if the shape functional $\cJ$ is defined only in $\cC^{m,\al}_\star\times\La$, then we will consider its extension (and still call it $\cJ$) to the whole $\cC^{m,\al}\times\La$ as done in \eqref{extension}.

Fix some $\vartheta\in(-\ve,\ve)$. By hypothesis, $\Om$ is a critical shape for $\cJ(\cdottone,0)$. We claim that $\Om$ is a critical shape for $\cJ(\cdottone,\vartheta)$ (that is, $g_{\Om}(\vartheta)\equiv 0$) as well. Indeed, for any $\xi\in C^{m+1,\al}(\pa\Om)$ the following chain of equalities holds true:
\begin{equation*}
\int_{\pa\Om}g_{\Om}(\vartheta)\ \xi=\restr{\frac{d}{dt}}{t=0} \cJ\left(\Om_{t\xi}, \vartheta \right)
= \restr{\frac{d}{dt}}{t=0} J\left(R_\vartheta(\Om_{t\xi})\right)
= \restr{\frac{d}{dt}}{t=0}J(\Om_{t\xi})= \int_{\pa\Om}g_\Om(0)\ \xi=0,
\end{equation*}
where, in the third equality, we used the fact that $J$ is invariant with respect to the group of rotations $\{R_\vartheta\}_{\vartheta\in\RR}$. Since $\xi$ was arbitrary, $g_\Om(\vartheta)\equiv 0$ as claimed. 
On the other hand, $R_{-\vartheta}(\Om)$ is also a critical shape for $\cJ(\cdottone,\vartheta)$. Indeed, for all $\xi\in C^{m+1,\al}(R_{-\vartheta}(\pa\Om))$ we have
\begin{equation*}
\left(R_{-\vartheta}(\Om)\right)_{t\xi}= R_{-\vartheta}\left(\Om_{t\xi\circ(R_{-\vartheta})}\right). 
\end{equation*}
In turn, this implies that 
\begin{equation*}
\begin{aligned}
\int_{R_{-\vartheta}(\pa\Om)} g_{R_{-\vartheta}(\Om)}(\vartheta)\ \xi 
= \restr{\frac{d}{dt}}{t=0}\cJ\left(\left(R_{-\vartheta}(\Om) \right)_{t\xi}  ,\vartheta \right)
= \restr{\frac{d}{dt}}{t=0} \cJ \left( R_{-\vartheta}\left(\Om_{t\xi\circ(R_{-\vartheta})}\right), \vartheta \right)\\
= \restr{\frac{d}{dt}}{t=0} J\left( \Om_{t\xi\circ(R_{-\vartheta})} \right)
=\int_{\pa\Om} g_\Om(0)\ \xi\circ(R_{-\vartheta})=0. 
\end{aligned}
\end{equation*}
Again, by the arbitrariness of $\xi$, we conclude that $g_{R_{-\vartheta}(\Om)}(\vartheta)\equiv 0$, as claimed. 

Let us briefly summarize what we have shown. For any $\vartheta\in\RR$, we have found two critical shapes for the shape functional $\cJ(\cdottone, \vartheta)$, namely $\Om$ and $R_{-\vartheta}(\Om)$. By Corollary \ref{corol rotations}, if $|\vartheta|$ is small enough, there exists a small function $\Psi(\vartheta)\in C^{m+1,\al}(\pa\Om)$ such that 
\begin{equation*}
R_{-\vartheta}(\Om)=\Om_{\Psi(\vartheta)}.     
\end{equation*}

Furthermore, by Lemma \ref{lem I} (respectively Lemma \ref{lem II}) there exist an open neighborhood $X'$ of $0\in C^{m+1,\al}(\pa\Om)$, a small positive real number $\ve>0$ and a $C^1$ map $\widetilde\xi: (-\ve,\ve)\to X'$ such that the set $\Om_{\widetilde\xi(\vartheta)}$ is a critical shape for the shape functional $\cJ(\cdottone, \vartheta)$. Moreover, for $(\xi,\vartheta)\in X'\times (-\ve,\ve)$, the set $\Om_{\xi}$ is a critical shape for $\cJ(\cdottone,\vartheta)$ if and only if $\xi=\widetilde\xi(\vartheta)$. In other words, for $|\vartheta|<\ve$, $\Om_{\widetilde\xi(\vartheta)}$ is the only critical shape for $\cJ(\cdottone,\vartheta)$. In turn, this implies that
\begin{equation*}
\Om=\Om_{\widetilde\xi(\vartheta)}= \Om_{\Psi(\vartheta)}=
\Om_{R_{-\vartheta}} \quad \tfor |\vartheta|<\ve.
\end{equation*}
Finally, for all $\vartheta>0$, there exist $k\in\NN\cup\{0\}$ and $\vartheta_1\in (0,\ve/2)$ such that $\vartheta= k \frac \ve 2 +\vartheta_1$. Thus,
\begin{equation*}
R_\vartheta(\Om)= \underbrace{R_{\ve/2}\circ \dots \circ R_{\ve/2}}_{\text{$k$ times}} \circ R_{\vartheta_1} (\Om) = \Om. 
\end{equation*}
The case $\vartheta<0$ then follows. We just showed that $\Om$ is invariant with respect to the subgroup $\{R_\vartheta\}_{\vartheta\in\RR}$, as claimed.

\section{An alternative take on Serrin's overdetermined problem and the Saint-Venant inequality}\label{sec 7}

Let $E(\om)$ denote the torsional rigidity of the open set $\om$, that is $E(\om):=\int_\om |\gr u|^2$, where $u$ is the solution to the following boundary value problem:
\begin{equation}\label{torsion}
-\De u=1 \quad \tin \om, \quad u=0\quad \ton\pa\om.    
\end{equation}
This defines a shape functional $E:\cC^{1,\al}\to\RR$. Moreover, it is known that the associated functional 
$e_\om(\xi):=E(\om_\xi)$ is well defined and Fr\'echet differentiable in a neighborhood of $0\in C^{1,\al}(\pa\om)$ for each $\om\in\cC^{2,\al}$. Let now $J: \cC^{2,\al} \to\RR$, 
\begin{equation}\label{lagrangian}
    J(\om):= E(\om)-\mu(|\om|-V_0).  
\end{equation}
This is the Lagrangian associated to the constrained maximization problem \ref{saint venant} (here $V_0$ is the value of the volume constraint and $\mu$ is the associated Lagrange multiplier). Moreover, if we set $j_\om(\xi):= J(\om_\xi)$ as before, a standard computation with the aid of the Hadamard formula (\cite[Theorem 5.2.2]{HP2018}) yields 
\begin{equation}\label{lagrangian der}
j_\om'(0)[\xi]= \int_{\pa\om} \left( |\gr u|^2 -\mu  \right) \xi.
\end{equation}
In other words, if $\Om$ is a critical shape for $J$, then we must have $|\gr u|^2\equiv \mu$ on $\pa\Om$, and thus $\Om$ is a solution to Serrin's overdetermined problem \ref{serrin}. It is known that, when the solution $u$ of \eqref{torsion} also satisfies $|\gr u|\equiv const$ on $\pa\Om$ even in some weak sense, then $\pa\Om$ is an analytic surface (see \cite{garofalo lewis, Vo, KN77}), so, in what follows, we will not care much about the regularity assumptions. 

Let us now compare problems \ref{saint venant} and \ref{serrin} in light of \eqref{lagrangian}-\eqref{lagrangian der}. If we set aside the assumptions on the regularity of $\Om$ we get:
\begin{enumerate}[label=\arabic*)]
    \item\label{sanv} Saint-Venant inequality: ``The ball (of volume $V_0$) is the only \textbf{maximizer} for the constrained maximization problem with Lagrangian $J$ among {\textbf{open sets}}." 
    \item\label{ser} Serrin's overdetermined problem: ``The ball (of volume $V_0$) is the only \textbf{critical shape} of $J$ among {\textbf{domains}}." 
\end{enumerate}
In other words, \ref{sanv} requires very strong assumptions on the variational behavior of $J$ at $\Om$ (namely, non-local ones, since it requires $\Om$ to be a \emph{global} maximizer), but makes no a priori assumptions on the connectedness of $\Om$. On the other hand, \ref{ser} just requires $\Om$ to be a critical shape, but connectedness is imposed. Indeed, as shape derivatives are local in nature, the family of critical shapes is closed under finite disjoint unions. In other words, the disjoint union of balls of volume $V_0$ is still a critical shape of $J$. In this sense, we can state that the connectedness assumption in \ref{ser} is sharp. 
As the following theorem shows, the same cannot be said for the assumption of $\Om$ being a global maximizer of \ref{sanv} 
\begin{theorem}
Let $\Om\in \cC^{3,\al}$ (in particular, $\Om$ is not necessarily connected) be a nondegenerate critical shape for the Lagrangian $J:\cC^{1,\al}_\star\to\RR$. Then $\Om$ is a ball. 
\end{theorem}
\begin{proof}
By Theorem \ref{mainthm II}, $\Om$ must be spherically symmetric. That is $\Om$ can be written as a (potentially infinite) disjoint union as follows
\begin{equation*}
    B\cup \bigcup_{i\in I} A_i,
\end{equation*}
where the $A_i$'s are annuli (spherical shells) centered at the origin and $B$ is either a ball centered at the origin or the empty set. We claim that $I=\emptyset$, that is, $\Om=B$ is a ball. To this end, assume that $I$ is not empty and consider a connected component $A_i$ of $\Om$. Since $A_i$ is an annulus, say $A_i=B_R\setminus \ol{B_r}$ ($0<r<R$), the solution to \eqref{torsion} can be computed explicitly. 
One can check that, for no positive value of the two radii $r<R$, the function $|\gr u|$ attains the same value on the two connected components $\pa B_R$ and $\pa B_r$ of $\pa A_i$. This is a contradiction. We conclude that $I$ must be empty and, thus, $\Om$ is a ball as claimed. 
\end{proof}
\begin{remark}
An analogous result can be given linking problems \ref{isoperimetric} and \ref{soap bubble} by considering the Lagrangian $J(\om):= |\om|-\mu(|\pa\om|-P_0)$.
\end{remark}

In what follows, we will discuss how this theory applies to the two-phase Serrin's problem (see \cite{CY1, CYisaac} for a local analysis of the family of nontrivial solutions to the two-phase Serrin's problem near concentric balls). Let us briefly recall the notation. Let $D$, $\om$ be two bounded open sets of $\rn$ that satisfy $\ol D\subset \om$. Moreover, set 
$\sg:= \sg_c \cX_D+ \cX_{\rn\setminus D}$, 
where $\sg_c>0$ is a given positive constant.

Along the same lines as before, let $E_D(\om)$ denote the two-phase torsional rigidity of the pair $(D,\om)$. That is, $E_D(\om):= \int_\om \sg |\gr u_D|^2$, where $u_D$ is the solution to the following boundary value problem:
\begin{equation}\label{two phase torsion}
    -\dv(\sg \gr u_D)=1\quad\tin\om, \quad u=0\quad \ton \pa\om.
\end{equation}
As before, let $J_D:\cC^{1,\al}\to\RR$, $J_D(\om):=E_D(\om)-\mu(|\om|-V_0)$ be the Lagrangian associated to the problem of maximizing $E_D$ under volume constraint. 
It is known that $\Om$ is a critical shape for $J_D$ if and only if $u_D$ also satisfies the following overdetermined condition:
\begin{equation}\label{oc}
|\gr u|\equiv \mu\quad \ton \pa\Om.
\end{equation}
We will refer to the overdetermined problem \eqref{two phase torsion}-\eqref{oc} as the \emph{two-phase Serrin's problem}.

Let $(D,\Om)$ be a solution to the two-phase Serrin's problem. In what follows we will discuss how the geometries of $D$ and $\Om$ are related. A first noteworthy result in this direction is due to Sakaguchi (proven, in a more general setting in \cite[Theorem 5.1]{Sak bessatsu}):

\begin{thm}
Let $B\subset\rn$ be a ball centered at the origin. Moreover, let $D\in\cC^2$ be an open set with finitely many connected components and let $B\setminus\ol D$ be connected. If $(D,B)$ is a solution to the two-phase Serrin's problem, then $D$ and $B$ are concentric balls. 
\end{thm}

The ``converse" does not hold. Indeed, when $D$ is a ball, it is known (see \cite{CYisaac}) that there exist symmetry-breaking solutions of the two-phase Serrin's problem for a discrete set of values of $\sg_c$. Moreover, the computations done in \cite{cava2018} show that these values are precisely the ones for which the ball $\Om_0$ is a \emph{degenerate} critical shape for $J_D$. We remark that the symmetry breaking solutions $\Om$ found in \cite{CY1} are not radially symmetric but only invariant with respect to a strictly smaller subgroup of rotations $\Ga\subsetneq SO(N)$, that is, $\Om$ only partially inherits the symmetry of $D$. The following direct application of Theorem \ref{mainthm I} states that the converse holds as well. 
\begin{theorem}\label{serrin two phase}
Let $D\in\cC^{0,1}$ and let $\Om\in\cC^{3,\al}$ be a nondegenerate critical shape for the Lagrangian $J_D$. If $D$ is invariant with respect to a one-parameter group of rotations $\{R_\vartheta\}_{\vartheta\in\RR}$, then $\Om$ is also invariant with respect to $\{R_\vartheta\}_{\vartheta\in\RR}$. In particular, if $D$ is a ball, then $\Om$ is also a ball concentric with $D$.
\end{theorem}
\begin{small}

\end{small}

\noindent
\textsc{
Mathematical Institute, Tohoku University, Aoba-ku, 
Sendai 980-8578, Japan}\\
\noindent
{\em Electronic mail address:}
cavallina.lorenzo.e6@tohoku.ac.jp


\begin{thebibliography}{99}

 
 


\bibitem[Al]
{Ale1958} \textsc{A.D.~Alexandrov}, {\em Uniqueness theorems for surfaces in the large V}. Vestnik Leningrad Univ., 13 (1958), 5--8 (English translation: Trans. Amer. Math. Soc., 21 (1962), 412--415). 


\bibitem[AP]{AP1983}
\textsc{A. Ambrosetti, G. Prodi}, {A Primer of Nonlinear Analysis}, Cambridge Univ. Press (1983).


\bibitem[AB]{ashbaugh benguria}
\textsc{M.S. Ashbaugh, R. Benguria}, {\em Proof of the Payne-P\'olya-Weinberger conjecture}. Bull. Amer. Math. Soc. 25 (1991), 19--29.





\bibitem[Br]{Br}
\textsc{F. Brock}, {\em An isoperimetric inequality for eigenvalues of the Stekloff problem}, Z.Angew. Math. Mech., 81 (2001), no. 1, 69--71.

\bibitem[BZ]{Burago Zalgaller}
\textsc{Y.D Burago \& V.A. Zalgaller}, Geometric Inequalities. (Translated from the Russian by
A.B. Sosinski\u{\i}.) Grundlehren der Mathematischen Wissenschaften 
(285). Springer Series in Soviet Mathematics. Springer-Verlag, Berlin,
1988.



\bibitem[Ca1]
{cava2018}\textsc{L.~Cavallina}, {\em Stability analysis of the two-phase torsional rigidity near a radial configuration}. Applicable Analysis, 98 (2019), no. 10, 1889–1900. https://doi.org/10.1080/00036811.2018.1478082



\bibitem[Ca2]{cava nondegenerate} \textsc{L. Cavallina}, {\em Nondegeneracy implies the existence of parametrized families of free boundaries}. arXiv:2109.12559




\bibitem[CY1]{CY1}
\textsc{L. Cavallina, T. Yachimura},
{\em On a two-phase Serrin-type problem and its numerical computation}, ESAIM: Control, Optimisation and Calculus of Variations (2020). https://doi.org/10.1051/cocv/2019048

\bibitem[CY2]{CYisaac}
\textsc{L. Cavallina, T. Yachimura},
{\em Symmetry breaking solutions for a two-phase overdetermined problem of Serrin-type}, Current Trends in Analysis, its Applications and Computation (Proceedings of the 12th ISAAC Congress, Aveiro, Portugal, 2019) Research Perspectives, Birkh\"auser (2022), 433--441. 


\bibitem[DL]{DL2019}
\textsc{M. Dambrine, J. Lamboley},
\emph{Stability in shape optimization with second variation}, Journal of Differential Equations 267 (2019), no. 5, 3009--3045,
https://doi.org/10.1016/j.jde.2019.03.033.

\bibitem[Dg]{DeG}
\textsc{E. De Giorgi}, {\em Sulla propriet\`a isoperimetrica dell'ipersfera, nella classe degli insiemi aventi
frontiera orientata di misura finita}, Atti Accad. Naz. Lincei. Mem. Cl. Sci. Fis. Mat. Nat. Sez.
I, 8 (1958), 33--44.

 \bibitem[DZ]
 {SG} \textsc{M.C.~Delfour, J.P.~Zol\'esio}, {Shapes and Geometries: Metrics, Analysis, Differential Calculus, and Optimization}. SIAM, Philadelphia (2001).





\bibitem[Fa]{Fa}
\textsc{G. Faber}, {\em Beweis, da{\ss} unter allen homogenen Membranen von gleicher Fl\"ache und gleicher Spannung die kreisf\"ormige den tiefsten Grundton gibt}, M\"unchen 1923, Sitzungsberichte: 1923, 8. https://publikationen.badw.de/de/003399311



\bibitem[GL]{garofalo lewis}
\textsc{N. Garofalo, J.L. Lewis}, \emph{A symmetry result related to some overdetermined boundary value problems}, Amer. J. Math. 111 (1989), no. 1, 9–-33.  

\bibitem[GS]{garofalo sartori}
\textsc{N. Garofalo, E. Sartori}, \emph{Symmetry in exterior boundary value problems for quasilinear elliptic equations via blow-up and a priori estimates}, Adv. Differential Equations 4 (1999), no. 2, 137-–161.

\bibitem[GT]{GT} \textsc{D.~Gilbarg, N.S.~Trudinger}, {Elliptic Partial Differential Equation of Second Order, second edition}. Springer (1983).



\bibitem[HP]
{HP2018} \textsc{A.~Henrot, M.~Pierre}, {Shape variation and optimization
(a geometrical analysis),} EMS Tracts in Mathematics,
Vol.28, European Mathematical Society (EMS),
Z\"urich, (2018).






\bibitem[Ka]{Ka}
\textsc{B. Kawohl}, {Rearrangements and convexity of level sets in PDE}, Lecture Notes in Mathematics, vol. 1150, Springer-Verlag, Berlin, 1985. MR 810619


\bibitem[KN]{KN77}\textsc{D.~Kinderlehrer, L.~Nirenberg}, 
{\em Regularity in free boundary problems},
Annali della Scuola Normale Superiore di Pisa - Classe di Scienze, Série 4, Tome 4 no. 2, (1977),  373--391.

\bibitem[Kr]{Kr}
\textsc{E. Krahn}, {\em {\"U}ber eine von Rayleigh formulierte Minimaleigenschaft des Kreises}, Math. Ann. 94, 97-–100 (1925). https://doi.org/10.1007/BF01208645








\bibitem[NP]{structure} \textsc{A.~Novruzi, M.~Pierre}, {\em Structure of shape derivatives}. Journal of Evolution Equations 2 (2002): 365--382.


\bibitem[Pa]{Palais 69} \textsc{R.S. Palais}, {\em The Morse Lemma for Banach Spaces}. Bulletin of the American Mathematical Society 75 (1969), 968--971.


\bibitem[Re]{reichel}
\textsc{W. Reichel}, \emph{Radial symmetry for elliptic boundary-value problems on exterior domains}, Arch. Rational Mech. Anal. 137 (1997) 381--394.

\bibitem[Sv]{SV} \textsc{B. de Saint-Venant}, {\em M\'emoire sur la torsion des prismes}, M\'emoires pr\'esent\'es par divers savants \`a l'Acad\'emie des Sciences, 14 (1856), 233--560.


\bibitem[Sa]{Sak bessatsu}
\textsc{S.~Sakaguchi},
{\em Two-phase heat conductors with a stationary isothermic surface and their related elliptic overdetermined problems}, RIMS K\^oky\^uroku Bessatsu B80 (2020), 113--132.



\bibitem[Se]
{Se1971} \textsc{J.~Serrin}, {\em A symmetry problem in potential theory}. Arch. Rat. Mech. Anal., 43 (1971), 304--318.

\bibitem[Sm]{smale} \textsc{S. Smale}, {\em Morse Theory and a non linear generalization of the Dirichlet problem}, Annals of Mathematics, Vol 80 No 2 (Sep. 1964), 382--396.

\bibitem[SZ]
{SZ1992} \textsc{J.~Sokolowski, J.P.~Zol\'esio}, {Introduction to Shape Optimization: Shape Sensitivity Analysis}, Springer Series in Computational Mathematics, 10, Springer--Verlag, Berlin, (1992).


\bibitem[Sz]{Sz}
\textsc{G. Szeg\H{o}}, {\em Inequalities for certain eigenvalues of a membrane of given area}, J. Rational Mech. Anal. 3 (1954), 343--356.



\bibitem[Ta]{Ta}
\textsc{G. Talenti}, {\em Elliptic equations and rearrangements}, Annali della Scuola Normale Superiore di Pisa - Classe di Scienze, S\'erie 4, Tome 3 (1976) no. 4, 697--718.

\bibitem[Tr]{tromba} \textsc{A. J. Tromba}, {\em A general approach to Morse Theory}, J. Diff. Geom., 12 (1977), 47--85.

\bibitem[Vo]
{Vo} \textsc{A.~L.~Vogel}, {\em Symmetry and regularity for general regions having a solution to certain overdetermined boundary value problems}, Atti Sem. Mat. Fis. Univ. Modena 40 (1992), no. 2, 443--484.

\bibitem[Wb1]{Wei}
\textsc{H. F. Weinberger}, {\em An isoperimetric inequality for the N-dimensional free membrane problem}, J. Rational Mech. Anal. 5 (1956), 633--636.

\bibitem[Wb2]{Wei preceding}
\textsc{H. F. Weinberger}, \emph{Remark on the preceding paper of Serrin}, Arch. Ration. Mech. Anal. 43 (1971), 319--320.

\bibitem[Ws]{Ws}
\textsc{R. Weinstock}, {\em Inequalities for a classical eigenvalue problem}, J. Rational Mech. Anal., 3 (1954), 745--753. 

\end{thebibliography}
\end{document}